\documentclass[conference]{IEEEtran}

\usepackage{comment}

\includecomment{detail} \excludecomment{nodetail}

\usepackage[sort,compress]{cite}
\usepackage{graphicx}

\usepackage[cmex10]{mathtools}
\usepackage{dsfont,amssymb,bm}

\usepackage[amsmath,thmmarks] {ntheorem}
\newtheorem {problem}       {Problem}
\newtheorem {property}      {Property}

\newtheorem {theorem}{Theorem}

\newtheorem {lemma}{Lemma}

\theoremstyle{plain}
\theorembodyfont{\normalfont}
\newtheorem {remark}{Remark}

\theoremstyle{plain}
\theoremheaderfont{\normalfont\itshape}
\theorembodyfont{\normalfont}

\theoremstyle{plain}
\theoremheaderfont{\normalfont\itshape}
\theorembodyfont{\normalfont}
\newtheorem *{proof}{Proof:}

\usepackage[T1]{fontenc}

\newcommand\VEC{\bm}
\newcommand\ALPHABET{\mathcal}
\newcommand\BLANK{\mathfrak {E}}
\newcommand\ASU{\mathrm{SU}}

\newcommand\IND{\mathds{1}}
\newcommand\PR {\mathds{P}}
\newcommand\EXP{\mathds{E}}

\newcommand\reals{\mathds{R}}

\newcommand\DEFINED{\coloneqq}

\begin{document}

\title {Structure of optimal strategies for remote
estimation over Gilbert-Elliott channel with feedback}

\author{\IEEEauthorblockN{Jhelum Chakravorty}
\IEEEauthorblockA{Electrical and Computer Engineering\\
McGill University,
Montreal, Canada\\
Email: jhelum.chakravorty@mail.mcgill.ca}
\and
\IEEEauthorblockN{Aditya Mahajan}
\IEEEauthorblockA{Electrical and Computer Engineering\\
McGill University,
Montreal, Canada\\
Email: aditya.mahajan@mcgill.ca}
}


\maketitle

\begin{abstract}
We investigate remote estimation over a Gilbert-Elliot channel with feedback.
We assume that the channel state is observed by the receiver and fed back to
the transmitter with one unit delay. In addition, the transmitter gets
\textsc{ack}/\textsc{nack} feedback for successful/unsuccessful transmission.
Using ideas from team theory, we establish the structure of optimal
transmission and estimation strategies and identify a dynamic program to
determine optimal strategies with that structure. We then consider
first-order autoregressive sources where the noise process has unimodal and
symmetric distribution. Using ideas from majorization theory, we show that
the optimal transmission strategy has a threshold structure and the optimal
estimation strategy is Kalman-like.
\end{abstract}

\section{Introduction}

\subsection{Motivation and literature overview}

We consider a remote estimation system in which a sensor/transmitter
observes a first-order Markov process and causally decides which observations
to transmit to a remotely located receiver/estimator. Communication is
expensive and takes place over a Gilbert-Elliot channel (which is used to
model channels with burst erasures). The channel has two states: \textsc{off}
state and \textsc{on} state. When the channel is in the \textsc{off} state, a
packet transmitted from the sensor to the receiver is dropped. When the
channel is in the \textsc{on} state, a packet transmitted from the sensor to
the receiver is received without error. We assume that the channel state is
causally observed at the receiver and is fed back to the transmitter with
one-unit delay. Whenever there is a successful reception, the receiver sends
an acknowledgment to the transmitter. The feedback is assumed to be noiseless. 

At the time instances when the receiver does not receive a packet (either
  because the sensor did not transmit or because the transmitted packet was
dropped), the receiver needs to estimate the state of the source process.
There is a fundamental trade-off between communication cost and
estimation accuracy. Transmitting all the time minimizes the estimation error
but incurs a high communication cost; not transmitting at all minimizes the
communication cost but incurs a high estimation error. 

The motivation of remote estimation comes from networked control systems. The
earliest instance of the problem was perhaps considered by Marschak~\cite{marschak1954}
in the context of information gathering in organizations. In recent years,
several variations of remote estimation has been considered. These include
models that consider idealized channels without packet drops~\cite{ImerBasar, Rabi2012, XuHes2004a, LipsaMartins:2011, NayyarBasarTeneketzisVeeravalli:2013, MH2012,
johanssonTAC13,  shi2015event} and models that consider
channels with i.i.d.\ packet drops~\cite{JC_AM_IFAC16, Shietal_arxiv}. 

The salient features of remote estimation are as follows:
\begin{enumerate}
  \item[(F1)] The decisions are made sequentially.
  \item[(F2)] The reconstruction/estimation at the receiver must be done
    with zero-delay.
  \item[(F3)] When a packet does get through, it is received without noise. 
\end{enumerate}
Remote estimation problems may be viewed as a special case of real-time
communication~\cite{Witsenhausen:1979, WalrandVaraiya:1983, Teneketzis:2006,
MT:real-time}. As in real-time communication, the key conceptual
difficulty is that the data available at the transmitter and the receiver is
increasing with time. Thus, the domain of the transmission and the estimation
function increases with time. 

To circumvent this difficulty one needs to identify sufficient statistics
for the data at the transmitter and the data at the receiver. In the
real-time communication literature, dynamic team theory (or decentralized
stochastic control theory) is used to identify such sufficient statistics as
well as to identify a dynamic program to determine the optimal transmission
and estimation strategies. Similar ideas are also used in remote-estimation
literature. In addition, feature (F3) allows one to further simplify the
structure of optimal transmission and estimation strategies. In particular,
when the source is a first-order autoregressive process, majorization theory
is used to show that the optimal transmission strategies is characterized
by a threshold~\cite{LipsaMartins:2011, NayyarBasarTeneketzisVeeravalli:2013,
MH2012, Shietal_arxiv, JC_AM_IFAC16}. In particular, it is optimal to
transmit when the instantaneous distortion due to not transmitting is greater
than a threshold. The optimal thresholds can be computed either using dynamic
programming~\cite{LipsaMartins:2011, NayyarBasarTeneketzisVeeravalli:2013} or
using renewal relationships~\cite{JC_AM_TAC15, JC_AM_IFAC16}.

All of the existing literature on remote-estimation considers either channels
with no packet drops or channels with i.i.d.\ packet drops. In this paper, we
consider packet drop channels with Markovian memory. We identify sufficient
statistics at the transmitter and the receiver. When the source is a
first-order autoregressive process, we show that threshold-based strategies
are optimal but the threshold depends on the previous state of the channel.

\subsection{The communication system}

\subsubsection{Source model} The source is a first-order time-homogeneous Markov
process $\{X_t\}_{t \ge 0}$, $X_t \in \ALPHABET X$. For ease of exposition, in
the first part of the paper we assume that $\ALPHABET X$ is a finite set. We
will later argue that a similar argument works when $\ALPHABET X$ is a general
measurable space. The transition probability matrix of the source is denoted
by $P$, i.e., for any $x, y \in \ALPHABET X$, 
\[
  P_{xy} \DEFINED \PR(X_{t+1} = y \mid X_t = x).
\]

\subsubsection{Channel model} The channel is a Gilbert-Elliott
channel~\cite{gilbert1960,elliott1963}. The channel state $\{S_t\}_{t \ge 0}$
is a binary-valued first-order time-homogeneous Markov process. We use the
convention that $S_t = 0$ denotes that the channel is in the \textsc{off}
state and $S_t = 1$ denotes that the channel is in the \textsc{on} state. The
transition probability matrix of the channel state is denoted by $Q$, i.e.,
for $r, s \in \{0,1\}$, 
\[
  Q_{rs} \DEFINED \PR(S_{t+1} = s | S_t = r).
\]

The input alphabet $\bar {\ALPHABET X}$ of the channel is $\ALPHABET X \cup
\{\BLANK\}$, where $\BLANK$ denotes the event that there is no transmission.
The channel output alphabet $\ALPHABET Y$ is $\ALPHABET X \cup \{\BLANK_0,
\BLANK_1\}$, where the symbols $\BLANK_0$ and $\BLANK_1$ are explained below.
At time $t$, the channel input is denoted by $\bar X_t$ and the channel output is
denoted by $Y_t$. 

The channel is a channel with state. In particular, for any realization
$(\bar x_{0:T}, s_{0:T}, y_{0:T})$ of $(\bar X_{0:T}, S_{0:T}, Y_{0:T})$, we
have that 
\begin{multline}
  \label{eq:channel}
  \PR(Y_t = y_t \mid \bar X_{0:t} = \bar x_{0:t}, S_{0:t} = s_{0:t}) 
  \\ =
  \PR(Y_t = y_t \mid \bar X_t = \bar x_t, S_t = s_t)
\end{multline}
and
\begin{multline}
  \label{eq:state}
  \PR(S_t = s_t \mid \bar X_{0:t} = \bar x_{0:t}, S_{0:t-1} = s_{0:t-1}) 
  \\ =
  \PR(S_t = s_t \mid S_{t-1} = s_{t-1}) = Q_{s_{t-1} s_t}
\end{multline}

Note that the channel output $Y_t$ is a deterministic function of the input
$\bar X_t$ and the state $S_t$. In particular, for any $\bar x \in \bar
{\ALPHABET X}$ and $s \in \{0,1\}$, the channel output $y$ is given as
follows:
\[
  y = \begin{cases}
    \bar x, & \text{if $\bar x \in \ALPHABET X$ and $s = 1$} \\
    \BLANK_1, & \text{if $\bar x= \BLANK$ and $s = 1$} \\
    \BLANK_0, & \text{if $s = 0$}
  \end{cases}
\]
This means that if there is a transmission (i.e., $\bar x \in \ALPHABET X$)
and the channel is on (i.e., $s = 1$), then the receiver observes $\bar x$.
However, if there is no transmission (i.e., $\bar x = \BLANK$) and the
channel is on (i.e., $s=1$), then the receiver observes~$\BLANK_1$, if the
channel is off, then the receiver observes~$\BLANK_0$. 

\subsubsection{The transmitter}

There is no need for channel coding in a remote-estimation setup. Instead,
the role of the transmitter is to determine which source realizations need to
be transmitted. Let $U_t \in \{0,1\}$ denote the transmitter's decision. We
use the convention that $U_t = 0$ denotes that there is no transmission
(i.e., $\bar X_t = \BLANK$) and $U_1 = 1$ denotes that there is transmission
(i.e., $\bar X_t = X_t$). 

Transmission is costly. Each time the transmitter transmits (i.e., $U_t =
1$), it incurs a cost of~$\lambda$.

\subsubsection{The receiver}

At time~$t$, the receiver generates an estimate $\hat X_t \in \ALPHABET X$ of
$X_t$. The quality of the estimate is determined by a distortion function $d
\colon \ALPHABET X  \times \ALPHABET X \to \reals_{\ge 0}$. 

\subsection{Information structure and problem formulation}

It is assumed that the receiver observes the channel state causally. Thus, the
information available at the receiver\footnote{We use superscript~$1$ to
  denote variables at the transmitter and superscript~$2$ to denote
variables at the receiver.} is
\[
  I^2_t = \{S_{0:t}, Y_{0:t} \}.
\]
The estimate $\hat X_t$ is chosen according to
\begin{equation} \label{eq:rx-1}
  \hat X_t = g_t(I^2_t) = g_t(S_{0:t}, Y_{0:t}),
\end{equation}
where $g_t$ is called the \emph{estimation rule} at time~$t$. The collection
$\VEC g \DEFINED (g_1, \dots, g_T)$ for all time is called the
\emph{estimation strategy}.

It is assumed that there is one-step delayed feedback from the receiver to
the transmitter.\footnote{Note that feedback requires two bits: the channel
state $S_t$ is binary and the channel output $Y_t$ can be communicated by
indicating whether $Y_t \in \ALPHABET X$ or not (i.e., transmitting an
\textsc{ack} or a \textsc{nack}).} Thus, the information available at the
transmitter is
\[
  I^1_t = \{X_{0:t}, U_{0:t-1}, S_{0:t-1}, Y_{0:t-1} \}.
\]
The transmission decision $U_t$ is chosen according to
\begin{equation}\label{eq:tx-1}
  U_t = f_t(I^1_t) = f_t(X_{0:t}, U_{0:t-1}, S_{0:t-1}, Y_{0:t-1}),
\end{equation}
where $f_t$ is called the \emph{transmission rule} at time~$t$. The collection
$\VEC f \DEFINED (f_1, \dots, f_T)$ for all time is called the
\emph{transmission strategy}.

The collection $(\VEC f, \VEC g)$ is called a \emph{communication strategy}.
The performance of any communication strategy $(\VEC f, \VEC g)$ is given by
\begin{equation} \label{eq:cost}
  J(\VEC f, \VEC g) = 
  \EXP\bigg[ \sum_{t=0}^T \lambda U_t + d(X_t, \hat X_t) \bigg]
\end{equation}
where the expectation is taken with respect to the joint measure on all
system variables induced by the choice of $(\VEC f, \VEC g)$. 

We are interested in the following optimization problem.
\begin{problem} \label{prob:finite}
  In the model described above,  identify a communication strategy $(\VEC
  f^*, \VEC g^*)$ that minimizes the cost $J(\VEC f, \VEC g)$ defined
  in~\eqref{eq:cost}.
\end{problem}

\section{Main results}\label{sec:main-results}

\subsection{Structure of optimal communication strategies}

Two-types of structural results are established in the real-time
communication literature: (i)~establishing that part of the data at
the transmitter is irrelevant and can be dropped without any loss of
optimality; (ii)~establishing that the common information between the
transmitter and the receiver can be ``compressed'' using a belief state. The
first structural results were first established by
Witsenhausen~\cite{Witsenhausen:1979} while the second structural results
were first established by Walrand Varaiya~\cite{WalrandVaraiya:1983}. 

We establish both types of structural results for remote estimation. First,
we show that $(X_{0:t-1}, U_{0:t-1})$ is irrelevant at the transmitter
(Lemma~\ref{lemma:irrelevant_info_Tx}); then, we use the common information
approach of~\cite{NMT:partial-history-sharing} and establish a belief-state
for the common information $(S_{0:t}, Y_{0:t})$ between the transmitter and
the receiver (Theorem~\ref{thm:structure}). 

\begin{lemma}\label{lemma:irrelevant_info_Tx}
  For any estimation strategy of the form~\eqref{eq:rx-1}, there is no loss
  of optimality in restricting attention to transmission strategies of the
  form
  \begin{equation} \label{eq:tx-2}
    U_t = f_t(X_t, S_{0:t-1}, Y_{0:t-1}).
  \end{equation}
\end{lemma}

The proof idea is similar to~\cite{Teneketzis:2006}. We show that $\{X_t,
S_{0:t-1}, Y_{0:t-1}\}_{t \ge 0}$ is a controlled Markov process controlled
by $\{U_t\}_{t \ge 0}$. 
\begin{nodetail} 
  See~\cite{preprint} for proof. 
\end{nodetail}
\begin{detail}
  See Section~\ref{sec:structure} for proof.
\end{detail}

Now, following~\cite{NMT:partial-history-sharing}, for any transmission
strategy $\VEC f$ of the form~\eqref{eq:tx-2} and any realization $(s_{0:T},
y_{0:T})$ of $(S_{0:T}, Y_{0:T})$, define $\varphi_t \colon \ALPHABET X \to
\{0, 1\}$ as
\[
  \varphi_t(x) = f_t(x, s_{0:t-1}, y_{0:t-1}), 
  \quad \forall x \in \ALPHABET X.
\]
Furthermore, define conditional probability measures~$\pi^1_t$ and $\pi^2_t$
on $\ALPHABET X$ as follows: for any $x \in \ALPHABET X$,
\begin{align*}
  \pi^1_t(x) &\DEFINED \PR^{\VEC f}(X_t = x \mid 
    S_{0:t-1} = s_{0:t-1}, Y_{0:t-1} = y_{0:t-1} ), 
  \\
  \pi^2_t(x) &\DEFINED \PR^{\VEC f}(X_t = x \mid 
  S_{0:t} = s_{0:t}, Y_{0:t} = y_{0:t} ).
  \\
\end{align*}
We call $\pi^1_t$ the \emph{pre-transmission belief} and $\pi^2$ the
\emph{post-transmission belief}. 
Note that when $(S_{0:T}, Y_{0:T})$ are random variables, then $\pi^1_t$ and
$\pi^2_t$ are also random variables which we denote by $\Pi^1_t$ and $\Pi^2_t$.

For the ease of notation, for any $\varphi \colon \ALPHABET X \to \{0,1\}$
and $i \in \{0, 1\}$, define the following:
\begin{itemize}
  \item $B_i(\varphi) = \{ x \in \ALPHABET X : \varphi(x) = i \}$.
  \item For any probability distribution $\pi$ on $\ALPHABET X$ and any
    subset $\ALPHABET A$ of $\ALPHABET X$, $\pi(\ALPHABET A)$ denotes
    $\sum_{x \in \ALPHABET A} \pi(x)$. 
  \item For any probability distribution $\pi$ on $\ALPHABET X$, $\xi =
    \pi|_{\varphi}$ means that $\xi(x) = \IND_{\{ \varphi(x) = 0
    \}} \pi(x)/\pi(B_0(\varphi))$.
\end{itemize}

\begin{lemma}\label{lemma:F1-F2}
  Given any transmission strategy $\VEC f$ of the
  form~\eqref{eq:tx-2}:
  \begin{enumerate}
    \item there exists a function~$F^1$ such that
      \begin{equation}
        \pi^1_{t+1} = F^1(\pi^2_t) = \pi^2_t P.
        \label{eq:pi-update-2} 
      \end{equation}
    \item there exists a function $F^2$ such that
      \begin{equation}
        \pi^2_t = F^2(\pi^1_t, \varphi_{t}, y_t).
        \label{eq:pi-update-1} 
      \end{equation}
      In particular,
      \begin{equation}
        \pi^2_t = \begin{cases}
          \delta_{y_t} & \mbox{if $y_t \in \ALPHABET X$} 
          \\
          \pi^1_t|_{\varphi_t}
          , & \mbox{if $y_t = \BLANK_1$} \\
          \pi^1_t , & \mbox{if $y_t = \BLANK_0$}.
        \end{cases}
        \label{eq:F2}
      \end{equation}
  \end{enumerate}
\end{lemma}

Note that in~\eqref{eq:pi-update-2}, we are treating $\pi^2_t$ as a
row-vector and in~\eqref{eq:F2}, $\delta_{y_t}$ denotes a Dirac measure
centered at $y_t$. The update equations~\eqref{eq:pi-update-2} and
\eqref{eq:pi-update-1} are standard non-linear filtering equations.
\begin{nodetail} 
  See~\cite{preprint} for proof. 
\end{nodetail}
\begin{detail}
  See Section~\ref{sec:structure} for proof.
\end{detail}

\begin{theorem} \label{thm:structure}
  In Problem~\ref{prob:finite}, we have that:
  \begin{enumerate} 
    \item \emph{Structure of optimal strategies:}
      There is no loss of optimality in restricting attention to optimal
      transmission and estimation strategies of the form:
      \begin{align}
        U_t &= f^*_t(X_t, S_{t-1}, \Pi^1_t), 
        \label{eq:tx-*}\\
        \hat X_t &= g^*_t(\Pi^2_t).
        \label{eq:rx-*}
      \end{align}
    \item \emph{Dynamic program:} Let $\Delta(\mathcal X)$ denote the space of probability distributions on $\mathcal X$. Define value functions $V^1_t \colon
      \{0,1\} \times \Delta(\ALPHABET X) \to \reals$ and $V^2_t \colon \{0,
      1\} \times \Delta(\ALPHABET X) \to \reals$ as follows. 
      \begin{gather}
        V^1_{T+1}(s, \pi^1) = 0, \\
        \intertext{and for $t \in \{T, \dots, 0\}$}
        \begin{lgathered}[\hsize]
          V^1_t(s, \pi^1) = \min_{\varphi \colon \ALPHABET X \to \{0, 1\}} 
          \Big\{ \lambda \pi^1(B_1(\varphi))
            \\
            \quad {} + 
            W^0_t(\pi^1, \varphi) \pi^1(B_0(\varphi))
            + \smashoperator{\sum_{x \in B_1(\varphi)} }
            W^1_t(\pi^1, \varphi,x) \pi^1(x)
          \Big\}
        \end{lgathered}
        \label{eq:dp-1}
        \\
        V^2_t(s, \pi^2) = \min_{\hat x \in \ALPHABET X}
        \sum_{x \in \ALPHABET X} d(x, \hat x) \pi^2(x)
        + V^1_{t+1}(s, \pi^2 P),
        \label{eq:dp-2}
      \end{gather}
     where, 
    \begin{align*}
      W^0_t(\pi^1, \varphi) &= 
      Q_{s0} V^2_t(0, \pi^1)
      +
      Q_{s1} V^2_t(1, \pi^1|_\varphi ),
      \\
      W^1_t(\pi^1, \varphi, x) &= 
      Q_{s0} V^2_t(0, \pi^1)
      +
      Q_{s1} V^2_t(1, \delta_x).
    \end{align*}

    Let $\Psi_t(s, \pi^1)$ denote the arg min of the right hand side
    of~\eqref{eq:dp-1}. Then, the optimal transmission strategy of the
    form~\eqref{eq:tx-*} is given by
    \[
      f^*_t(\cdot, s, \pi^1) = \Psi_t(s, \pi^1).
    \]
    Furthermore, the optimal estimation strategy of the form~\eqref{eq:rx-*}
    is given by
    \begin{equation} \label{eq:rx-form}
      g^*_t(\pi^2) = \arg \min_{\hat x \in \ALPHABET X} \sum_{x \in \ALPHABET X}
      d(x,\hat x) \pi^2(x).
    \end{equation}
  \end{enumerate}
\end{theorem}

The proof idea is as follows. Once we restrict attention to transmission
strategies of the form~\eqref{eq:tx-2}, the information structure is partial
history sharing~\cite{NMT:partial-history-sharing}. Thus, one can use the
common information approach of~\cite{NMT:partial-history-sharing} and obtain
the structure of optimal strategies. 
\begin{nodetail}
  See~\cite{preprint} for proof. 
\end{nodetail}
\begin{detail}
  See Section~\ref{sec:structure} for proof.
\end{detail}

\begin{remark}
  The first term in~\eqref{eq:dp-1} is the expected communication cost, the
  second term is the expected cost-to-go when the transmitter does not
  transmit, and the third term is the expected cost-to-go when the transmitter
  transmits. The first term in~\eqref{eq:dp-2} is the expected distortion and
  the second term is the expected cost-to-go.
\end{remark}

\begin{remark}
  Although the above model and result are stated for sources with finite
  alphabets, they extend naturally to general state spaces (including
  Euclidean spaces) under standard technical assumptions.
  See~\cite{Yuksel:2013} for details.
\end{remark}

\subsection{Optimality of threshold-based strategies for autoregressive source}

In this section, we consider a first-order autoregressive source $\{X_t\}_{t
\ge 0}$, $X_t \in \reals$, where the initial state $X_0 = 0$ and for $t \ge
0$, we have that
\begin{equation}\label{eq:AR1}
  X_{t+1} = a X_t + W_t,
\end{equation}
where $a \in \reals$ and $W_t \in \reals$ is distributed according to a
symmetric and unimodal distribution with probability density function~$\mu$.
Furthermore, the per-step distortion is given by $d(X_t - \hat X_t)$, where
$d(\cdot)$ is a even function that is increasing on $\reals_{\ge 0}$. The
rest of the model is the same as before. 

For the above model, we can further simplify the result of
Theorem~\ref{thm:structure}. See Section~\ref{sec:proof_AR-Theorem} for the proof.
\begin{theorem}\label{thm:threshold}
  For a first-order autoregressive source with symmetric and unimodal disturbance, 
  \begin{enumerate}
    \item \emph{Structure of optimal estimation strategy:} The optimal
      estimation strategy is given as follows: $\hat X_0 = 0$, and for $t \ge
      0$,
      \begin{equation}
        \hat X_{t} = \begin{cases}
          a \hat X_{t-1}, & \mbox{if $Y_t \in \ALPHABET \{\BLANK_0, \BLANK_1\}$} \\
          Y_t, & \mbox{if $Y_t \in \reals$}
        \end{cases}
        \label{eq:rx-a}
      \end{equation}

    \item \emph{Structure of optimal transmission strategy:} There exist
      threshold functions $k_t \colon \{0, 1\} \to \reals_{\ge 0}$ such that
      the following transmission strategy is optimal:
      \begin{equation}
        f_t(X_t, S_{t-1}, \Pi^1_t) = \begin{cases}
          1, & \mbox{if $|X_t - a \hat X_{t-1}| \ge k_t(S_{t-1})$} \\
          0, & \mbox{otherwise}.
        \end{cases}
        \label{eq:tx-a}
      \end{equation}
  \end{enumerate}
\end{theorem}

\begin{remark}
  As long as the receiver can distinguish between the events $\BLANK_0$
  (i.e., $S_t = 0$)  and $\BLANK_1$ (i.e., $U_t = 0$ and $S_t = 1$), the
  structure of the optimal estimator does not depend on the channel state
  information at the receiver. 
\end{remark}

\begin{remark}
  It can be shown that under the optimal strategy, $\Pi^2_t$ is symmetric and
  unimodal around $\hat X_t$ and, therefore, $\Pi^1_t$ is symmetric and
  unimodal around $a \hat X_{t-1}$. Thus,
  the transmission and estimation strategies in Theorem~\ref{thm:threshold}
  depend on the pre- and post-transmission beliefs only through their means.
\end{remark}

\begin{remark}
  Recall that the distortion function is even and increasing. Therefore, the
  condition $|X_t - a\hat X_{t-1}| \ge k_t(S_{t-1})$ can be written as
  $d(X_t - a\hat X_{t-1}) \ge \tilde k_t(S_{t-1}) \DEFINED d(k_t(S_{t-1}))$.
  Thus, the optimal strategy is to transmit if the per-step distortion due to
  not transmitting is greater than a threshold. 
\end{remark}

\begin{detail}
  \section{Proof of the structural results} \label{sec:structure}

  \subsection{Proof of Lemma~\ref{lemma:irrelevant_info_Tx}}

  Arbitrarily fix the estimation strategy $\VEC g$ and consider the \emph{best
  response} strategy at the transmitter. We will show that $\tilde I^1_t
  \DEFINED (X_t, S_{0:t-1}, Y_{0:t-1})$ is an information state at the
  transmitter. 

  Given any realization $(x_{0:T}, s_{0:T}, y_{0:T}, u_{0:T})$ of the system
  variables $(X_{0:T}, S_{0:T}, Y_{0:T}, U_{0:T})$, define $i^1_t = (x_{0:t},
    s_{0:t-1},
  y_{0:t-1}, \allowbreak u_{0:t-1})$ and $\tilde \imath^1_t = (x_t, s_{0:t-1},
  y_{0:t-1})$. Now, for any $\breve \imath^1_{t+1} = (\breve x_{t+1},
  \breve s_{0:t}, \breve y_{0:t}) = (\breve x_{t+1}, \breve s_t, \breve y_t,
  \breve \imath^1_t)$, we use the shorthand $\PR(\tilde
  \imath^1_{t+1} | \tilde \imath^1_{0:t}, u_{0:t})$ to denote $\PR(\tilde
    I^1_{t+1} = \breve \imath^1_{t+1} | \tilde I^1_{0:t} = \tilde
  i^1_{0:t}, U_{0:t} = u_{0:t})$. Then,
  \begin{align}
    \hskip 2em & \hskip -2em
    \PR(\breve \imath^1_{t+1} | i^1_{t}, u_{t}) =
    \PR(\breve x_{t+1}, \breve s_{t}, \breve y_{t}, \breve \imath^1_t | x_{0:t}, s_{0:t-1}, y_{0:t-1}, u_{0:t}) \notag \\
    &\stackrel{(a)}=
    \PR(\breve x_{t+1}, \breve s_{t}, \breve y_{t}, \breve \imath^1_t | x_{0:t}, \bar x_{0:t}, s_{0:t-1}, y_{0:t-1}, u_{0:t}) \notag \\
    &\stackrel{(b)}= 
    \PR(\breve x_{t+1} | x_t) \PR(\breve y_t | \bar x_t, \breve s_{t}) 
    \PR(\breve s_t | s_{t-1}) \IND_{\{ \breve \imath^1_t = \tilde \imath^1_t \}}
    \notag \\
    &=
    \PR(\breve \imath^1_{t+1} | \tilde \imath^1_t, u_t)
    \label{eq:Tr-Markov}
  \end{align}
  where we have added $\bar x_{0:t}$ in the conditioning in $(a)$ because
  $\bar x_{0:t}$ is a deterministic function of $(x_{0:t}, u_{0:t})$ and
  $(b)$ follows from the source and the channel models.
  By marginalizing~\eqref{eq:Tr-Markov}, we get that
  for any $\breve \imath^2_t = (\breve s_t, \breve y_t, \breve \imath^1_t)$,
  we have
  \begin{equation}
    \PR(\breve \imath^2_t | i^1_t, u_t) =
    \PR(\breve \imath^2_t | \tilde \imath^1_t, u_t) 
    \label{eq:Tr-Markov-2}
  \end{equation}
  
  Now, let $c(X_t, U_t, \hat X_t) = \lambda U_t + d(X_t, \hat X_t)$ denote the
  per-step cost. Recall that $\hat X_t = g_t(I^2_t)$. Thus,
  by~\eqref{eq:Tr-Markov-2}, we get that
  \begin{equation}
    \EXP[c(X_t, U_t, \hat X_t) | i^1_t, u_t]
    = 
    \EXP[c(X_t, U_t, \hat X_t) | \tilde \imath^1_t, u_t ].
    \label{eq:Tr-Markov-3}
  \end{equation}

  Eq.~\eqref{eq:Tr-Markov} shows that $\{\tilde I^1_t\}_{t \ge 0}$ is a
  controlled Markov process controlled by $\{U_t\}_{t \ge 0}$.
  Eq.~\eqref{eq:Tr-Markov-3} shows that $\tilde I^1_t$ is sufficient for
  performance evaluation. Hence, by Markov decision
  theory~\cite{KumarVaraiya:1986}, there is no loss of optimality in
  restricting attention to transmission strategies of the
  form~\eqref{eq:tx-2}.

  \subsection{Proof of Lemma~\ref{lemma:F1-F2}}

  Consider
  \begin{align}
    \pi^1_{t+1}(x_{t+1}) &= \PR(x_{t+1}|s_{0:t}, y_{0:t}) \notag \\
    &= \sum_{x_t \in \ALPHABET X} \PR(x_{t+1} | x_t) \PR(x_t | s_{0:t}, y_{0:t})
    \notag \\
    &= \sum_{x_t \in \ALPHABET X} P_{x_t x_{t+1}} \pi^2_t(x_t) = \pi^2_t P
    \label{eq:F1-update}
  \end{align}
  which is the expression for $F^1(\cdot)$.

  For $F^2$, we consider the three cases separately. 
  For $y_t \in \ALPHABET X$, we have
  \begin{equation}
    \pi^2_t(x) = \PR(X_t = x | s_{0:t}, y_{0:t}) = \IND_{\{ x = y_t \}}.
    \label{eq:pi-2-1}
  \end{equation}

  For $y_t \in \{\BLANK_0, \BLANK_1\}$, we have
  \begin{align}
    \pi^2_t(x) &= \PR(X_t = x | s_{0:t}, y_{0:t})  \notag \\
    &= \frac { \PR(X_t = x, y_t, s_t | s_{0:t-1}, y_{0:t-1}) }
             { \PR(y_t, s_t | s_{0:t-1}, y_{0:t-1}) }
    \label{eq:pi-2-2}
  \end{align}

  Now, when $y_t = \BLANK_0$, we have that
  \begin{align}
    \hskip 2em & \hskip -2em
    \PR(x_t, y_t, s_t | s_{0:t-1}, y_{0:t-1}) =
    \PR(y_t | x_t, \varphi_t(x_t), s_t) Q_{s_{t-1}s_t} \pi^1_t(x_t)
    \notag \\
    &\stackrel{(a)}=
    \begin{cases}
      Q_{s_{t-1} 1} \pi^1_t(x_t)
          , & \text{if $\varphi_t(x_t) = 0$ and $s_t = 1$} \\
      0   , & \text{otherwise}
    \end{cases}
    \label{eq:pi-2-1b}
  \end{align}
  where~$(a)$ is obtained from the channel model.
  Substituting~\eqref{eq:pi-2-1b} in~\eqref{eq:pi-2-2} and canceling
  $Q_{s_{t-1} 1} \IND_{\{s_t = 1\}}$ from the numerator and the denominator, we
  get (recall that this is for the case when $y_t = \BLANK_0$),
  \begin{align}
    \pi^2_t(x) 
    &= \frac { \IND_{\{\varphi_t(x) = 0\}} \pi^1_t(x) }
             { \pi^1_t( B_0(\varphi)) }.
      \label{eq:pi-2-3a}
  \end{align}

  Similarly, when $y_t = \BLANK_1$, we have that
  \begin{align}
    \hskip 2em & \hskip -2em
    \PR(x_t, y_t, s_t | s_{0:t-1}, y_{0:t-1}) =
    \PR(y_t | x_t, \varphi_t(x_t), s_t) Q_{s_{t-1}s_t} \pi^1_t(x_t)
    \notag \\
    &\stackrel{(b)}=
    \begin{cases}
      Q_{s_{t-1} 0} \pi^1_t(x_t)
          , & \text{if $s_t = 0$} \\
      0   , & \text{otherwise}
    \end{cases}
    \label{eq:pi-2-2b}
  \end{align}
  where~$(b)$ is obtained from the channel model.
  Substituting~\eqref{eq:pi-2-2b} in~\eqref{eq:pi-2-2} and canceling
  $Q_{s_{t-1}0} \IND_{\{s_t = 0\}}$ from the numerator and the denominator, we
  get (recall that this is for the case when $y_t = \BLANK_1$),
  \begin{equation}
    \pi^2_t(x) 
      =
      \pi^1_t(x).
      \label{eq:pi-2-3b}
  \end{equation}
  By combining~\eqref{eq:pi-2-1},
  \eqref{eq:pi-2-3a} and~\eqref{eq:pi-2-3b}, we get~\eqref{eq:F2}.

  \subsection{Proof of Theorem~\ref{thm:structure}}

  Once we restrict attention to transmission strategies of the
  form~\eqref{eq:tx-2}, the information structure is partial history
  sharing~\cite{NMT:partial-history-sharing}. Thus, one can use the common
  information approach of~\cite{NMT:partial-history-sharing} and obtain the
  structure of optimal strategies. 

  Following~\cite{NMT:partial-history-sharing}, we split the information
  available at each agent into a ``common information'' and ``local
  information''. Common information is the information available to all
  decision makers in the future; the remaining data at the decision maker is
  the local information. Thus, at the transmitter, the common information
  is $C^1_t \DEFINED \{S_{0:t-1}, Y_{0:t-1}\}$ and the local information is
  $L^1_t \DEFINED X_t$. Similarly, at the receiver, the common information is
  $C^2_t \DEFINED \{S_{0:t}, Y_{0:t}\}$ and the local information is $L^2_t
  \DEFINED \emptyset$. When the transmitter makes a decision, the state
  (sufficient for input output mapping) of the system is $(X_t, S_{t-1})$; when
  the receiver makes a decision, the state of the system is $(X_t, S_t)$.
  By~\cite[Proposition 1]{NMT:partial-history-sharing}, we get that the
  sufficient statistic $\Theta^1_t$ for the common information at the
  transmitter is 
  \[
    \Theta^1_t(x,s) = \PR(X_t = x, S_{t-1} = s | S_{0:t-1}, Y_{0:t-1}),
  \]
  and the sufficient statistic $\Theta^2_t$ for the common information at the receiver is
  \[
    \Theta^2_t(x,s) = \PR(X_t = x, S_t = s | S_{0:t}, Y_{0:t}).
  \]
  Note that $\Theta^1_t$ is equivalent to $(\Pi^1_t, S_{t-1})$ and $\Theta^2_t$
  is equivalent to $(\Pi^2_t, S_t)$. Therefore, by~\cite[Theorem 2]
  {NMT:partial-history-sharing}, there is no loss of optimality in
  restricting attention to transmission strategies of the form~\eqref{eq:tx-*}
  and estimation strategies of the form
  \begin{equation}
    \hat X_t = g_t(S_t, \Pi^2_t).
    \label{eq:rx-alt}
  \end{equation}
  Furthermore, the dynamic program of~\ref{thm:structure} follows
  from~\cite[Theorem 3]{NMT:partial-history-sharing}.

  Note that the right hand side of~\eqref{eq:dp-2} implies that $\hat X_t$
  does not depend on $S_t$. Thus, instead of~\eqref{eq:rx-alt}, we can
  restrict attention to estimation strategy of the form~\eqref{eq:rx-*}.
  Furthermore, the optimal estimation strategy is given by~\eqref{eq:rx-form}.

\end{detail}

\section{Proof of optimality of threshold-based strategies for autoregressive
source}\label{sec:proof_AR-Theorem}

\subsection{A change of variables}

Define a process $\{Z_t\}_{t \ge 0}$ as follows: $Z_0 = 0$ and for $t \ge 0$,
\[
  Z_{t} = \begin{cases}
    a Z_{t-1}, & \mbox{if $Y_t \in \{\BLANK_0, \BLANK_1\}$} \\
    Y_t, & \mbox{if $Y_t \in \ALPHABET X$} 
  \end{cases}
\]

Note that $Z_t$ is a function of $Y_{0:t-1}$. 
Next, define processes $\{E_t\}_{t \ge 0}$, $\{E^+_t\}_{t \ge 0}$, and $\{\hat
E_t\}_{t \ge 0}$ as follows:
\[
  E_t \DEFINED X_t - a Z_{t-1}, \quad
  E^+_t \DEFINED X_t - Z_t, \quad
  \hat E_t \DEFINED \hat X_t - Z_t
\]
The processes $\{E_t\}_{t \ge 0}$ and $\{ E^+_t \}_{t \ge 0}$ are related
as follows: $E_0 = 0$, $E^+_0 = 0$, and for $t \ge 0$
\begin{align*}
  E^+_t &= \begin{cases}
    E_t, & \mbox{if $Y_t \in \{\BLANK_0, \BLANK_1\}$} \\
    0  , & \mbox{if $Y_t \in \ALPHABET X$}
  \end{cases}
  \shortintertext{and}
  E_{t+1} &= a E^+_t + W_t.
\end{align*}

Since $X_t - \hat X_t = E^+_t - \hat E_t$, we have that $d(X_t - \hat X_t) =
d(E^+_t - \hat E_t)$.

It turns out that it is easier to work with the processes $\{E_t\}_{t \ge
0}$, $\{ E^+_t \}_{t \ge 0}$, and $\{ \hat E_t \}_{t \ge 0}$ rather than
$\{X_t\}_{t \ge 0}$ and $\{\hat X_t \}_{t \ge 0}$. 

Next, redefine the pre- and post-transmission beliefs in terms of the error
process. With a slight abuse of notation, we still denote the (probability
density) of the pre- and post-transmission beliefs as $\pi^1_t$ and
$\pi^2_t$. In particular, $\pi^1_t$ is the conditional pdf of $E_t$ given
$(s_{0:t-1}, y_{0:t-1})$ and $\pi^2_t$ is the conditional pdf of $E^+_t$ given
$(s_{0:t}, y_{0:t})$. 

Let $H_t \in \{\BLANK_0, \BLANK_1, 1\}$ denote the event whether the
transmission was successful or not. In particular,
\[
  H_t = \begin{cases}
    \BLANK_0, & \mbox{if $Y_t = \BLANK_0$} \\
    \BLANK_1, & \mbox{if $Y_t = \BLANK_1$} \\
    1, & \mbox{if $Y_t \in \reals$}.
  \end{cases}
\]
We use $h_t$ to denote the realization of $H_t$. Note that $H_t$ is a
deterministic function of $U_t$ and $S_t$. 

The time-evolutions of $\pi^1_t$ and $\pi^2_t$ is similar to
Lemma~\ref{lemma:F1-F2}. In particular, we have
\begin{lemma} \label{lemma:F1-F2a}
  Given any transmission strategy $\VEC f$ of the form~\eqref{eq:tx-1}:
  \begin{enumerate} 
    \item there exists a function $F^1$ such that
      \begin{equation} \label{eq:update-2a}
        \pi^1_{t+1} = F^1(\pi^2_t).
      \end{equation}
      In particular,
      \begin{equation}\label{eq:F1-h}
        \pi^1_{t+1} = \begin{cases}
                      \tilde \pi^2_t \star \mu, & \mbox{if $y_t \in \{\BLANK_0,\BLANK_1\}$}\\
                      \mu, & \mbox{if $y_t \in \reals$},
                      \end{cases}
      \end{equation}
      where $\tilde \pi^2_t$ given by $ \tilde \pi^2_t(e) \DEFINED (1/|a|) \pi^2_t(e/a)$ is the conditional probability density of $aE^+_t$, $\mu$ is the probability density function of $W_t$ and $\star$ is the convolution operation.

    \item there exists a function $F^2$ such that
      \begin{equation} \label{eq:update-1a}
        \pi^2_t = F^2(\pi^1_t, \varphi_t, h_t).
      \end{equation}
      In particular,
      \begin{equation}\label{eq:F2-h}
        \pi^2_t = \begin{cases}
          \delta_{0}, & \mbox{if $h_t = 1$} \\
          \pi^1_t|_{\varphi_t}, & \mbox{if $h_t = \BLANK_1$} \\
          \pi^1_t, & \mbox{if $h_t = \BLANK_0$}.
        \end{cases}
      \end{equation}
  \end{enumerate}
\end{lemma}

The key difference between Lemmas~\ref{lemma:F1-F2} and~\ref{lemma:F1-F2a}
(and the reason that we work with the error process $\{E_t\}_{t \ge 0}$
rather than $\{X_t\}_{t \ge 0}$) is that the  function $F^2$
in~\eqref{eq:update-1a} depends on $h_t$ rather than $y_t$. Consequently, the
dynamic program of Theorem~\ref{thm:structure} is now given by
  \begin{gather}
    V^1_{T+1}(s, \pi^1) = 0, \label{eq:dp-0a}\\
    \intertext{and for $t \in \{T, \dots, 0\}$}
    \begin{lgathered}[\hsize]
      V^1_t(s, \pi^1) = \min_{\varphi \colon \reals \to \{0, 1\}} 
      \Big\{ \lambda \pi^1(B_1(\varphi))
        \\
        \quad {} + 
        W^0_t(\pi^1, \varphi) \pi^1(B_0(\varphi)) 
        +
        W^1_t(\pi^1, \varphi) \pi^1(B_1(\varphi)) 
      \Big\}
    \end{lgathered}
    \label{eq:dp-1a}
    \\
    V^2_t(s, \pi^2) = D(\pi^2)
    + V^1_{t+1}(s, F^1(\pi^2) ),
    \label{eq:dp-2a}
  \end{gather}
 where, 
\begin{align*}
  W^0_t(\pi^1, \varphi) &= 
  Q_{s0} V^2_t(0, \pi^1)
  +
  Q_{s1} V^2_t(1, \pi^1|_\varphi ),
  \\
  W^1_t(\pi^1, \varphi) &= 
  Q_{s0} V^2_t(0, \pi^1)
  +
  Q_{s1} V^2_t(1, \delta_0),
  \\
  D(\pi^2) &= \min_{\hat e \in \reals}
  \int_{\reals} d(e - \hat e) \pi^2(e) de.
\end{align*}

Again, note that due to the change of variables, the expression for $W^1_t$
does not depend on the transmitted symbol. Consequently, the expression for
$V^1_t$ is simpler than that in Theorem~\ref{thm:structure}.

\subsection{Symmetric unimodal distributions and their properties}\label{sec:ASU-dist-prop}

A probability density function $\pi$ on reals is said to be \emph{symmetric
and unimodal ($\ASU$)} around $c \in \reals$ if for any $x \in \reals$,
$\pi(c-x) = \pi(c+x)$ and $\pi$ is non-decreasing in the interval $(-\infty,
c]$ and non-increasing in the interval $[c, \infty)$. 

Given $c \in \reals$, a prescription $\varphi \colon \reals \to \{0, 1\}$
is called \emph{threshold based around $c$} if there exists $k \in
\reals$ such that
\[
  \varphi(e) = \begin{cases} 
    1, & \mbox{if $|e - c| \ge k$} \\
    0, & \mbox{if $|e - c| < k$}.
  \end{cases}
\]
Let $\mathcal F(c)$ denote the family of all threshold-based prescription
around~$c$.

Now, we state some properties of symmetric and unimodal distributions..

\begin{property}\label{prop:1}
  If $\pi$ is $\ASU(c)$, then
  \[
    c \in \arg \min_{\hat e \in \reals} \int_\reals d(e - \hat e) \pi(e) de.
  \]
\end{property}
For $c=0$, the above property is a special case of~\cite[Lemma
12]{LipsaMartins:2011}. The result for general $c$ follows from a change of
variables.

\begin{property}\label{prop:2}
  If $\pi^1$ is $\ASU(0)$ and $\varphi \in \mathcal F(0)$, then for any $h
  \in \{\BLANK_0, \BLANK_1, 1\}$, $F^2(\pi^1, \varphi, h)$ is $\ASU(0)$. 
\end{property}
\begin{proof}
  We prove the result for each $h \in \{\BLANK_0, \BLANK_1, 1 \}$ separately.
  Recall the update of $\pi^1$ given by~\eqref{eq:F2-h}. For $h_t = \BLANK_0$,
  $\pi^2 = \pi^1$ and hence $\pi^2$ is $\ASU(0)$. For $h_t = \BLANK_1$, $\pi^2
  = \pi^1|_{\varphi}$; if $\varphi \in \mathcal F(0)$, then $\pi^1(x) \IND_{\{
  \varphi(x) = 0 \}}$ is $\ASU(0)$ and hence $\pi^1$ is $\ASU(0)$. For $h_t =
  1$, $\pi^2 = \delta_0$, which is $\ASU(0)$. 
\end{proof}

\begin{property}\label{prop:3}
  If $\pi^2$ is $\ASU(0)$, then $F^1(\pi^2)$ is also $\ASU(0)$.
\end{property}
\begin{proof}
  Recall that $F^1$ is given by~\eqref{eq:F1-h}. The property follows from
  the fact that convolution of symmetric and unimodal distributions is
  symmetric and unimodal.
\end{proof}

\subsection{$\ASU$ majorization and its properties}\label{sec:ASU-major-prop}

For any set $\ALPHABET A$, let $\mathcal I_{\ALPHABET A}$ denote its
indicator function, i.e., $\mathcal I_{\ALPHABET A}(x)$ is $1$ if $x \in
\ALPHABET X$, else $0$. 

Let $\ALPHABET A$ be a measurable set of finite Lebesgue measure, its
\emph{symmetric rearrangement} $\ALPHABET A^\sigma$ is the open interval centered
around origin whose Lebesgue measure is same as $\ALPHABET A$.

Given a function $\ell \colon \reals \to \reals$, its super-level set at
level $\rho$, $\rho \in \reals$, is $\{ x \in \reals : \ell(x) > \rho \}$.
The \emph{symmetric decreasing rearrangement} $\ell^\sigma$ of $\ell$ is a
symmetric and decreasing function whose level sets are the same as $\ell$,
i.e., 
\[
  \ell^\sigma(x) = \int_{0}^\infty 
  \mathcal I_{\{z \in \reals : \ell(z) > \rho\}^\sigma}(x) d\rho.
\]

Given two probability density functions $\xi$ and $\pi$ over $\reals$, $\xi$
\emph{majorizes} $\pi$, which is denoted by $\xi \succeq_m \pi$, if for all
$\rho \ge 0$, 
\[
  \int_{|x| \ge \rho} \xi^\sigma(x) dx \ge 
  \int_{|x| \ge \rho} \pi^\sigma(x) dx .
\]

Given two probability density functions $\xi$ and $\pi$ over $\reals$, $\xi$
\emph{SU majorizes} $\pi$, which we denote by $\xi \succeq_a \pi$,
if $\xi$ is $\ASU$ and $\xi$ majorizes $\pi$.

Now, we state some properties of SU majorization
from~\cite{LipsaMartins:2011}.

\begin{property} \label{prop:4}
  For any $\xi \succeq_a \pi$, where $\xi$ is $\ASU(c)$ and for any
  prescription $\varphi$, let $\theta \in \mathcal F(c)$ be a
  threshold-based prescription
  such that
  \[
    \xi(B_i(\theta)) = \pi(B_i(\varphi)), \quad i \in \{0, 1\}.
  \]
  Then,
  \(
    \xi|_\theta \succeq_a \pi|_\varphi.
  \)
  Consequently, for any $h \in \{\BLANK_0, \BLANK_1, 1\}$, 
  \[
    F^2(\xi, \theta, h) \succeq_a F^2(\pi, \varphi, h).
  \]
\end{property}
For $c = 0$, the result follows from~\cite[Lemma 7 and
8]{LipsaMartins:2011}. The result for general $c$ follows from change of
variables.

\begin{property}\label{prop:5}
  For any $\xi \succeq_m \pi$, $F^1(\xi) \succeq_a F^1(\pi)$. 
\end{property}
This follows from~\cite[Lemma 10]{LipsaMartins:2011}.

Recall the definition of $D(\pi^2)$ given after~\eqref{eq:dp-2a}.

\begin{property}\label{prop:6}
  If $\xi \succeq_a \pi$, then
  \[
    D(\pi) \ge D(\pi^\sigma) \ge D(\xi^\sigma) = D(\xi).
  \]
\end{property}
This follows from~\cite[Lemma 11]{LipsaMartins:2011}. 

\subsection{Qualitative properties of the value function and optimal strategy}


\begin{lemma} \label{lemma:property}
  The value functions~$V^1_t$ and $V^2_t$
  of~\eqref{eq:dp-0a}--\eqref{eq:dp-2a}, satisfy the following property. 
  \begin{enumerate}
    \item [(P1)] For any $i \in \{1,2\}$, $s \in \{0, 1\}$, $t \in \{0,
      \dots, T\}$, and pdfs $\xi^i$ and $\pi^i$ such that $\xi^i \succeq_a
      \pi^i$, we have that $V^i_t(s,\xi^i) \le V^i_t(s,\pi^i)$. 
  \end{enumerate}

  Furthermore, the optimal strategy satisfies the following properties.  For
  any $s \in \{0, 1\}$ and $t \in \{0, \dots, T\}$:
  \begin{enumerate}
    \item[(P2)] if $\pi^1$ is $\ASU(c)$, then there exists a prescription
      $\varphi_t \in \mathcal F(c)$ that is optimal. In general, $\varphi_t$
      depends on $\pi^1$. 
    \item[(P3)] if $\pi^2$ is $\ASU(c)$, then the optimal estimate $\hat E_t$
      is $c$.
  \end{enumerate}
\end{lemma}

\begin{proof}
  We proceed by backward induction. $V^1_{T+1}(s,\pi^1)$ trivially satisfies
  the (P1). This forms the basis of induction. Now assume that $V^1_{t+1}(s,
  \pi^1)$ also satisfies (P1). For $\xi^2 \succeq_a \pi^2$, we have that
  \begin{align}
    V^2_t(s, \pi^2) &= D(\pi^2) + V^1_{t+1}(s, F^1(\pi^2)) \notag \\
    &\stackrel{(a)}\ge D(\xi^2) + V^1_{t+1}(s, F^1(\xi^2)) \notag \\
    &= V^2_t(s, \xi^2),
    \label{eq:induction-2}
  \end{align}
  where $(a)$ follows from Properties~\ref{prop:5} and~\ref{prop:6} and the
  induction hypothesis. Eq.~\ref{eq:induction-2} implies that $V^2_t$ also
  satisfies (P1). 

  Now, consider $\xi^1 \succeq_a \pi^1$. Let $\varphi$ be the optimal
  prescription at $\pi^1$. Let $\theta$ be the threshold-based prescription
  corresponding to $\varphi$ as defined in Property~\ref{prop:3}. By
  construction,
  \[
    \pi^1(B_0(\varphi)) = \xi^1(B_0(\theta))
    \quad\text{and}\quad
    \pi^1(B_1(\varphi)) = \xi^1(B_1(\theta)).
  \]
  Moreover, from Property~\ref{prop:3} and~\eqref{eq:induction-2},
  \[
    W^0_t(\pi^1, \varphi) \ge W^0_t(\xi^1, \theta)
    \quad\text{and}\quad
    W^1_t(\pi^1, \varphi) \ge W^1_t(\xi^1, \theta).
  \]
  Combining the above two equations with~\eqref{eq:dp-1a}, we get
  \begin{align}
    V^1_t(s, \pi^1) &= \lambda \pi^1(B_1(\varphi)) 
    + W^0(\pi^1, \varphi) \pi^1(B_0(\varphi)) \notag \\
    & \quad + W^1(\pi^1, \varphi) \pi^1(B_1(\varphi)) \notag \\
    &\ge
      \lambda \xi^1(B_1(\theta )) 
    + W^0(\xi^1, \theta ) \xi^1(B_0(\theta )) \notag \\
    & \quad + W^1(\xi^1, \theta ) \xi^1(B_0(\theta )) \notag \\
    &\ge V^1_t(s, \xi^1)
    \label{eq:induction-1}
  \end{align}
  where the last inequality follows by minimizing over all $\theta$.
  Eq.~\eqref{eq:induction-1} implies that $V^1_t$ also satisfies (P1).
  Hence, by the principle of induction, (P1) is satisfied for all time.

  The argument in~\eqref{eq:induction-1} also implies (P2). Furthermore, (P3)
  follows from Property~\ref{prop:1}.
\end{proof}

\subsection{Proof of Theorem~\ref{thm:threshold}}
We first prove a weaker version of the structure of optimal transmission
strategies. In particular, there exist threshold functions $\tilde k_t
\colon \{0, 1\} \times \Delta(\reals) \to \reals_{\ge 0}$ such that the
following transmission strategy is optimal:
\begin{equation}
  f_t(X_t, S_{t-1}, \Pi^1_t) = \begin{cases}
    1, & \mbox{if $|X_t - a Z_{t-1}| \ge \tilde k_t(S_{t-1}, \Pi^1_t)$} \\
        0, & \mbox{otherwise}.
      \end{cases}
      \label{eq:tx-b}
\end{equation}
or, equivalently, in terms of the $\{E_t\}_{t \ge 0}$ process:
\begin{equation}
  f_t(E_t, S_{t-1}, \Pi^1_t) = \begin{cases}
    1, & \mbox{if $|E_t| \ge \tilde k_t(S_{t-1}, \Pi^1_t)$} \\
        0, & \mbox{otherwise}.
      \end{cases}
      \label{eq:tx-c}
\end{equation}

We prove~\eqref{eq:tx-c} by induction. Note that $\pi^1_0 = \delta_0$ which
is $\ASU(0)$. Therefore, by (P2), there exists a threshold-based
prescription $\varphi_0 \in \mathcal F(0)$ that is optimal. This forms the
basis of induction. Now assume that until time $t-1$, all prescriptions are
in $\mathcal F(0)$. By Properties~\ref{prop:2} and~\ref{prop:3}, $\Pi^1_t$
is $\ASU(0)$. Therefore, by (P2), there exists a threshold-based
prescription $\varphi_t \in \mathcal F(0)$ that is optimal. This proves the
induction step and, hence, by the principle of induction, threshold-based
prescriptions of the form~\eqref{eq:tx-c} are optimal for all time.
Translating the result back to $\{X_t\}_{t \ge 0}$, we get that
threshold-based prescriptions of the form~\eqref{eq:tx-b} are optimal.

Observe that Properties~\ref{prop:2} and~\ref{prop:3} also imply that for
all~$t$, $\Pi^2_t$ is $\ASU(0)$. Therefore, by Property~\ref{prop:1}, the
optimal estimate $\hat E_t = 0$. Recall that $\hat E_t = \hat X_t - Z_t$.
Thus, $\hat X_t = Z_t$. This proves the first part of
Theorem~\ref{thm:threshold}.

To prove that there exist optimal transmission strategies where the
thresholds do not depend on $\Pi^1_t$, we fix the estimation strategy to be
of the form~\eqref{eq:rx-a} and consider the problem of finding the best
transmission strategy at the sensor. This is a single-agent
(centralized) stochastic control problem and the optimal solution is given
by the following dynamic program:
\begin{align}
  J_{T+1}(e,s) &= 0 \label{eq:dp-J0}\\
  \intertext{and for $t \in \{T, \dots, 0\}$}
  J_t(e,s) &= \min\{ J^0_t(e,s), J^1_t(e,s) \}
\end{align}
where 
\begin{align}
  J^0_t(e,s) &= d(e) + Q_{s0}\EXP_W[J_{t+1}(ae + W,0)] \notag \\
  &\quad 
  + Q_{s1} \EXP_W[J_{t+1}(ae + W, 1) ],
  \\
  J^1_t(e,s) &= \lambda + Q_{s0}d(e) + Q_{s0}\EXP_W[J_{t+1}(ae + W,0)] \notag \\
  &\quad 
  + Q_{s1} \EXP_W[J_{t+1}(W, 1) ],
  \label{eq:dp-J2}
\end{align}

We now use the results of~\cite{JC-AM-OR16} to show that the value function
even and increasing on $\reals_{\ge 0}$ (abbreviated to EI).

The results of~\cite{JC-AM-OR16} rely on stochastic dominance. Given two
probability density functions $\xi$ and $\pi$ over $\reals_{\ge 0}$, $\xi$
\emph{stochastically dominates} $\pi$, which we denote by $\xi \succeq_s
\pi$, if
\[
  \int_{x \ge y} \xi(x) dx \ge \int_{x \ge y} \pi(x) dx, 
  \quad \forall y \in \reals_{\ge 0}.
\]

Now, we show that dynamic program~\eqref{eq:dp-J0}--\eqref{eq:dp-J2}
satisfies conditions (C1)--(C3) of~\cite[Theorem 1]{JC-AM-OR16}. In
particular, we have: Condition (C1) is satisfied because the per-step cost
functions $d(e)$ and $\lambda + Q_{s0}d(e)$ are EI. Condition (C2) is
satisfied because the probability density~$\mu$ of $W_t$ is even, which
implies that for any $e \in \reals_{\ge 0}$, 
\[  
  \int_{w \in \reals} \mu(ae+w) dw = \int_{w \in \reals} \mu(-ae+w) dw.
\]
Now, to check condition (C3), define for $e \in \reals$ and $y \in \reals_{\ge 0}$, 
\begin{align*}
  M^0(y|e) &= \int_{y}^{\infty} \mu(ae + w) dw 
  + \int_{-\infty}^{-y} \mu(ae + w) dw
  \\
  &=  1 - \int_{-y}^y \mu(ae + w) dw ,
  \\
  M^1(y|e) &= \int_{y}^{\infty} \mu(w) dw 
  + \int_{-\infty}^{-y} \mu(w) dw.
\end{align*}
$M^1(y|e)$ does not depend on $e$ and is thus trivially even and increasing
in~$e$. Since $\mu$ is even, $M^0(y|e)$ is even in~$e$. We show that
$M^0(y|e)$ is increasing in $e$ for $e \in \reals_{\ge 0}$ later (see
Lemma~\ref{lemma:M0-inc}).

Since conditions (C1)--(C3) of~\cite[Theorem 1]{JC-AM-OR16} are satisfied,
we have that for any $s \in \{0, 1\}$, $J_t(e,s)$ is even in~$e$ and
increasing for $e \in \reals_{\ge 0}$. Now, observe that
\begin{multline*}
  J^0(e, s) - J^1(e,s) = (1 - Q_{s0}) d(e) + Q_{s1} \EXP_W[ J_{t+1}(ae + W,
  1)] \\
  - \lambda - Q_{s1} \EXP_W[J_{t+1}(W, 1)]
\end{multline*}
which is even in $e$ and increasing in $e \in \reals_{\ge 0}$. Therefore, for any fixed $s \in \{0,1\}$, 
the set $A$ of $e$ in which $J^0_t(e, s) - J^1_t(e,s) \le 0$ is convex and
symmetric around the origin, i.e., a set of the form $[-k_t(s), k_t(s)]$. Thus,
there exist a $k_t(\cdot)$ such that the action $u_t = 0$ is optimal for $e \in
[-k_t(s), k_t(s)]$. This, proves the structure of the optimal transmission
strategy.

\begin{lemma} \label{lemma:M0-inc}
  For any $y \in \reals_{\ge 0}$, $M^0(y|e)$ is increasing in $e$, $e \in
  \reals_{\ge 0}$. 
\end{lemma}
\begin{proof}
To show that
$M^0(y|e)$ is increasing in~$e$ for $e \in \reals_{\ge 0}$, it sufficies to
show that $1 - M^0(y|e) = \int_{-y}^y \mu(ae + w) dw$ is decreasing in~$e$ for
$e \in \reals_{\ge 0}$. Consider a change of variables $x = ae + w$. Then,
\begin{equation}
  1 - M^0(y|e) = \int_{-y}^y \mu(ae + w) dw = 
  \int_{-y - ae }^{y - ae} \mu(x) dx
  \label{eq:int-mu}
\end{equation}
Taking derivative with respect to~$e$, we get that
\begin{equation}
  \frac{ \partial M^0(y|e) } { \partial e}
  = 
  a [ \mu(y - ae) -  \mu(-y - ae) ]
  \label{eq:diff-mu}
\end{equation}
Now consider the following cases:
\begin{itemize}
  \item If $a > 0$ and $y > ae  > 0$, then the right hand side
    of~\eqref{eq:diff-mu} equals $a [ \mu(y - ae) - \mu(y + ae) ]$, which is
    positive.
  \item If $a > 0$ and $ae > y > 0$, then the right hand side
    of~\eqref{eq:diff-mu} equals $a [ \mu(ae - y) - \mu(ae + y)]$, which is
    positive. 
  \item If $a < 0$ and $y > |a|e > 0$, then the right hand side
    of~\eqref{eq:diff-mu} equals $|a|\, [\mu(y - |a|e) - \mu(y + |a|e)]$,
    which is positive.
  \item If $a < 0$ and $|a|e > y > 0$, then the right hand side
    of~\eqref{eq:diff-mu} equals $|a|\, [\mu(|a|e - y) - \mu(|a|e + y)]$,
    which is positive.
\end{itemize}
Thus, in all cases, $M^0(y|e)$ is increasing in $e$, $e \in \reals_{\ge 0}$.
\end{proof}

\section{Conclusion}
In this paper, we studied remote estimation over a Gilbert-Elliot channel
with feedback. We assume that the channel state is observed by the receiver
and fed back to the transmitter with one unit delay. In addition, the
transmitter gets \textsc{ack}/\textsc{nack} feedback for
successful/unsuccessful transmission. Using ideas from team theory, we
establish the structure of optimal transmission and estimation strategies and
identify a dynamic program to determine optimal strategies with that
structure. We then consider first-order autoregressive sources where the
noise process has unimodal and symmetric distribution. Using ideas from
majorization theory, we show that the optimal transmission strategy has a
threshold structure and the optimal estimation strategy is Kalman-like.

A natural question is how to determine the optimal thresholds. For finite
horizon setup, these can be determined using the dynamic program
of~\eqref{eq:dp-J0}--\eqref{eq:dp-J2}. For inifinite horizon setup, we expect
that the optimal threshold will not depend on time. We believe that it should
be possible to evalute the performance of a generic threshold based strategy
using an argument similar to the renewal theory based argument presented
in~\cite{JC_AM_TAC15} for channels without packet drops.

\bibliographystyle{IEEEtran}
\bibliography{IEEEabrv,isit17_ref} 

\begin{thebibliography}{10}
\providecommand{\url}[1]{#1}
\csname url@samestyle\endcsname
\providecommand{\newblock}{\relax}
\providecommand{\bibinfo}[2]{#2}
\providecommand{\BIBentrySTDinterwordspacing}{\spaceskip=0pt\relax}
\providecommand{\BIBentryALTinterwordstretchfactor}{4}
\providecommand{\BIBentryALTinterwordspacing}{\spaceskip=\fontdimen2\font plus
\BIBentryALTinterwordstretchfactor\fontdimen3\font minus
  \fontdimen4\font\relax}
\providecommand{\BIBforeignlanguage}[2]{{%
\expandafter\ifx\csname l@#1\endcsname\relax
\typeout{** WARNING: IEEEtran.bst: No hyphenation pattern has been}%
\typeout{** loaded for the language `#1'. Using the pattern for}%
\typeout{** the default language instead.}%
\else
\language=\csname l@#1\endcsname
\fi
#2}}
\providecommand{\BIBdecl}{\relax}
\BIBdecl

\bibitem{marschak1954}
J.~Marschak, ``Towards an economic theory of organization and information,''
  \emph{Decision processes}, vol.~3, no.~1, pp. 187--220, 1954.

\bibitem{ImerBasar}
O.~C. Imer and T.~Basar, ``Optimal estimation with limited measurements,''
  \emph{Joint 44the IEEE Conference on Decision and Control and European
  Control Conference}, vol.~29, pp. 1029--1034, 2005.

\bibitem{Rabi2012}
M.~Rabi, G.~Moustakides, and J.~Baras, ``Adaptive sampling for linear state
  estimation,'' \emph{SIAM Journal on Control and Optimization}, vol.~50,
  no.~2, pp. 672--702, 2012.

\bibitem{XuHes2004a}
Y.~Xu and J.~P. Hespanha, ``Optimal communication logics in networked control
  systems,'' in \emph{Proceedings of 43rd IEEE Conference on Decision and
  Control}, vol.~4, 2004, pp. 3527--3532.

\bibitem{LipsaMartins:2011}
G.~M. Lipsa and N.~C. Martins, ``Remote state estimation with communication
  costs for first-order {LTI} systems,'' \emph{IEEE Transactions on Automatic
  Control}, vol.~56, no.~9, pp. 2013--2025, 2011.

\bibitem{NayyarBasarTeneketzisVeeravalli:2013}
A.~Nayyar, T.~Basar, D.~Teneketzis, and V.~V. Veeravalli, ``Optimal strategies
  for communication and remote estimation with an energy harvesting sensor,''
  \emph{IEEE Transactions on Automatic Control}, vol.~58, no.~9, pp.
  2246--2260, 2013.

\bibitem{MH2012}
A.~Molin and S.~Hirche, ``An iterative algorithm for optimal event-triggered
  estimation,'' in \emph{4th IFAC Conference on Analysis and Design of Hybrid
  Systems (ADHS'12)}, 2012, pp. 64--69.

\bibitem{johanssonTAC13}
J.~Wu, Q.~S. Jia, K.~H. Johansson, and L.~Shi, ``Event-based sensor data
  scheduling: Trade-off between communication rate and estimation quality,''
  \emph{IEEE Transactions on Automatic Control}, vol.~58, no.~4, pp.
  1041--1046, April 2013.

\bibitem{shi2015event}
D.~Shi, L.~Shi, and T.~Chen, \emph{Event-Based State Estimation: A Stochastic
  Perspective}.\hskip 1em plus 0.5em minus 0.4em\relax Springer, 2015, vol.~41.

\bibitem{JC_AM_IFAC16}
J.~Chakravorty and A.~Mahajan, ``Remote state estimation with packet drop,'' in
  \emph{6th {IFAC} Workshop on Distributed Estimation and Control in Networked
  Systems}, Sep 2016.

\bibitem{Shietal_arxiv}
R.~Xiaoqiang, W.~Junfeng, J.~K. Henrik, S.~Guodong, and S.~Ling, ``Infinite
  horizon optimal transmission power control for remote state estimation over
  fading channels,'' \emph{arxiv: 1604.08680v1 [cs.SY]}, Apr 29 2016.

\bibitem{Witsenhausen:1979}
H.~S. Witsenhausen, ``On the structure of real-time source coders,''
  \emph{BSTJ}, vol.~58, no.~6, pp. 1437--1451, July-August 1979.

\bibitem{WalrandVaraiya:1983}
J.~C. Walrand and P.~Varaiya, ``Optimal causal coding-decoding problems,''
  \emph{{IEEE} Trans. Inf. Theory}, vol.~29, no.~6, pp. 814--820, Nov. 1983.

\bibitem{Teneketzis:2006}
D.~Teneketzis, ``On the structure of optimal real-time encoders and decoders in
  noisy communication,'' \emph{{IEEE} Trans. Inf. Theory}, pp. 4017--4035, Sep.
  2006.

\bibitem{MT:real-time}
A.~Mahajan and D.~Teneketzis, ``Optimal design of sequential real-time
  communication systems,'' \emph{{IEEE} Trans. Inf. Theory}, vol.~55, no.~11,
  pp. 5317--5338, Nov. 2009.

\bibitem{JC_AM_TAC15}
J.~Chakravorty and A.~Mahajan, ``Fundamental limits of remote estimation of
  {M}arkov processes under communication constraints,'' \emph{IEEE Transactions
  on Automatic Control}, 2017 (to appear).

\bibitem{gilbert1960}
E.~N. Gilbert, ``Capacity of a burst-noise channel,'' \emph{Bell System
  Technical Journal}, vol.~39, no.~5, pp. 1253--1265, 1960.

\bibitem{elliott1963}
E.~O. Elliott, ``Estimates of error rates for codes on burst-noise channels,''
  \emph{Bell System Technical Journal}, vol.~42, no.~5, pp. 1977--1997, 1963.

\bibitem{NMT:partial-history-sharing}
A.~Nayyar, A.~Mahajan, and D.~Teneketzis, ``Decentralized stochastic control
  with partial history sharing: A common information approach,'' \emph{{IEEE}
  Trans. Autom. Control}, vol.~58, no.~7, pp. 1644--1658, jul 2013.

\bibitem{Yuksel:2013}
S.~Yuksel, ``On optimal causal coding of partially observed {Markov} sources in
  single and multiterminal settings,'' \emph{{IEEE} Trans. Inf. Theory},
  vol.~59, no.~1, pp. 424--437, 2013.

\bibitem{KumarVaraiya:1986}
P.~R. Kumar and P.~Varaiya, \emph{Stochastic Systems: Estimation,
  Identification and Adaptive Control}.\hskip 1em plus 0.5em minus 0.4em\relax
  Upper Saddle River, NJ, USA: Prentice-Hall, Inc., 1986.

\bibitem{JC-AM-OR16}
J.~Chakravorty and A.~Mahajan, ``On evenness and monotonicity of value
  functions and optimal strategies in markov decision processes,''
  \emph{submitted to Operations Research Letters}, 2016.

\end{thebibliography}

\end{document}